\documentclass{amsart}

\usepackage{graphicx}
\usepackage{amssymb}
\usepackage{geometry}

\usepackage[colorlinks=true,urlcolor=blue, citecolor=red,linkcolor=blue,
linktocpage,pdfpagelabels, bookmarksnumbered,bookmarksopen]{hyperref}
\usepackage[hyperpageref]{backref}

\newtheorem{lemma}{Lemma}[section]
\newtheorem{theorem}[lemma]{Theorem}
\newtheorem{corollary}[lemma]{Corollary}
\newtheorem{proposition}[lemma]{Proposition}
\newtheorem{open}{Open problem}

\theoremstyle{definition}
\newtheorem{definition}[lemma]{Definition}
\newtheorem{remark}[lemma]{Remark}
\newtheorem*{ack}{Acknowledgements}

\numberwithin{equation}{section}

\title[Principal frequencies and inradius]{On principal frequencies and inradius\\ in convex sets}

\author[Brasco]{Lorenzo Brasco}
\address{Dipartimento di Matematica e Informatica
\newline\indent
Universit\`a degli Studi di Ferrara
\newline\indent
Via Machiavelli 35, 44121 Ferrara, Italy}
\email{lorenzo.brasco@unife.it}


\dedicatory{To Michelino Brasco, master craftsman and father, on the occasion of his 70th birthday}

\begin{document}

\begin{abstract}
We generalize to the case of the $p-$Laplacian an old result by Hersch and Protter. Namely, we show that it is possible to estimate from below the first eigenvalue of the Dirichlet $p-$Laplacian of a convex set in terms of its inradius. We also prove a lower bound in terms of isoperimetric ratios and we briefly discuss the more general case of Poincar\'e-Sobolev embedding constants. Eventually, we highlight an open problem.
\end{abstract}

\maketitle
\begin{center}
\begin{minipage}{8cm}
\small
\tableofcontents
\end{minipage}
\end{center}

\section{Introduction}

\subsection{Overview}
For every open set $\Omega\subset\mathbb{R}^N$, we consider its {\it principal frequency} or {\it first eigenvalue of the Laplacian with Dirichlet conditions}, defined by
\[
\lambda(\Omega)=\inf_{u\in C^\infty_0(\Omega)\setminus \{0\}} \frac{\displaystyle \int_\Omega |\nabla u|^2\,dx}{\displaystyle \int_\Omega |u|^2\,dx}.
\]
We recall that, whenever the completion $\mathcal{D}^{1,2}_0(\Omega)$ of $C^\infty_0(\Omega)$ with respect to the norm $\|\nabla u\|_{L^2(\Omega)}$ is compactly embedded into\footnote{For example, this happens if $\Omega$ is bounded or has finite $N-$dimensional Lebesgue measure.} $L^2(\Omega)$, the number $\lambda(\Omega)$ coincides with the smallest $\lambda\in\mathbb{R}$ such that the boundary value problem
\[
-\Delta u=\lambda\,u,\quad \mbox{ in }\Omega,\qquad u=0,\quad \mbox{ on }\partial\Omega,
\]
does admit a nontrivial solution $u\in \mathcal{D}^{1,2}_0(\Omega)$.
\par
For general sets, the explicit determination of $\lambda(\Omega)$ can be a challenging task. It is thus important to look for sharp estimates on $\lambda(\Omega)$ in terms on simpler quantities, typically of geometric flavour. The most celebrated instance of such an estimate is the so-called {\it Faber-Krahn inequality}. This asserts that $\lambda(\Omega)$ can be estimated from below by a negative power of the $N-$dimensional measure of $\Omega$. Precisely, we have
\begin{equation}
\label{FK}
\lambda(\Omega)\ge \left(|B|^\frac{2}{N}\,\lambda(B)\right)\,\frac{1}{|\Omega|^\frac{2}{N}},
\end{equation}
where $B$ is any $N-$dimensional ball. Equality \eqref{FK} is sharp in the sense that the dimensional constant $|B|^\frac{2}{N}\,\lambda(B)$ is attained whenever $\Omega$ is itself a ball (actually, this is the only possibility, up to sets of zero capacity). 
\par
In despite of its elegance, sharpness and simplicity, the lower bound dictated by \eqref{FK} loses its interest for open sets such that
\[
|\Omega|=+\infty\qquad \mbox{ and }\qquad \lambda(\Omega)>0.
\]
This happens for example for the infinite slab $\Omega=\mathbb{R}^{N-1}\times (0,1)$. 
\par
For such cases, it could be natural to ask whether a lower bound on $\lambda(\Omega)$ can be given in terms of the {\it inradius} $R_\Omega$, i.e. the radius of the largest open ball contained in $\Omega$. In other words, we can ask whether we can have an inequality like
\begin{equation}
\label{inra}
\frac{C}{R_\Omega^2}\le \lambda(\Omega).
\end{equation}
The power $-2$ on $R_\Omega$ is imposed by scale invariance, once it is observed that $\lambda(\Omega)$ has the physical dimensions ``{\it length to the power $-2$}''. However, an estimate like \eqref{inra} can not be true for general open sets, in dimension $N\ge 2$. Indeed, it is sufficient to consider the set
\[
\Omega=\mathbb{R}^N\setminus \mathbb{Z}^N.
\] 
It is easy to see that $R_\Omega<+\infty$, while $\lambda(\Omega)=\lambda(\mathbb{R}^N)=0$, since points have zero capacity in $\mathbb{R}^N$, if $N\ge 2$. 
\par
However, if we impose further geometric restrictions on the open set $\Omega$, then it is possible to prove \eqref{inra}.
An old result due to Hersch (see \cite{He}) shows that for an open {\it convex} set $\Omega\subset\mathbb{R}^2$, it holds 
\begin{equation}
\label{hersch}
\left(\frac{\pi}{2}\right)^2\,\frac{1}{R_\Omega^2}\le \lambda(\Omega).
\end{equation}
The inequality is sharp and it is strict among bounded convex sets. The proof by Hersch is based on a method that he called ``{\it \'evaluation par d\'efaut\,}''.
Later on, Protter generalized this result to higher dimensions by using the same technique, see \cite[page 68]{Pr}. 
\par
We also point out that the Hersch-Protter estimate has been recently generalized in \cite[Theorem 5.1]{BuGuMa} to the {\it anisotropic case}, i.e. to the case of 
\[
\lambda_H(\Omega)=\inf_{u\in C^\infty_0(\Omega)\setminus \{0\}} \frac{\displaystyle \int_\Omega H(\nabla u)^2\,dx}{\displaystyle \int_\Omega |u|^2\,dx},
\]
where $H:\mathbb{R}^N\to [0,+\infty)$ is any norm. In this case, the definition of inradius has to be suitably adapted, in order to take into account the anisotropy $H$.
\begin{remark}[More general sets I]
We have already observed that \eqref{inra} can not be true in general.
However, the planar case $N=2$ is peculiar and well-studied: in this case, if $\Omega$ is simply connected, then it is possible to prove \eqref{inra}, but the main open issue in this case is the determination of the sharp constant $C$. The first result in this direction is due to Hayman \cite{Ha}. We refer to \cite{BC} for a review of this kind of results.
\par
Actually, Osserman in \cite{Os} showed that \eqref{inra} still holds for planar sets with finite connectivity, the constant $C$ depending on the connectivity $k$ and degenerating as $k$ goes to $\infty$ (this is in perfect accordance with the above example of $\mathbb{R}^2\setminus \mathbb{Z}^2$). The result by Osserman has then been improved by Croke in \cite{Cr}.
\par 
For the higher dimensional case $N\ge 3$, some results for classes of open sets more general than convex ones have been given by Hayman \cite[Theorem 2]{Ha} and Taylor \cite[Theorem 3]{Ta}. 
\end{remark}

\subsection{The results of this paper}

We now fix an exponent $1<p<+\infty$, then for an open set $\Omega\subset\mathbb{R}^N$, we introduce the quantity
\[
\lambda_p(\Omega)=\inf_{u\in C^\infty_0(\Omega)\setminus \{0\}} \frac{\displaystyle \int_\Omega |\nabla u|^p\,dx}{\displaystyle \int_\Omega |u|^p\,dx}.
\]
As in the quadratic case $p=2$, whenever the completion $\mathcal{D}^{1,p}_0(\Omega)$ of $C^\infty_0(\Omega)$ with respect to the norm $\|\nabla u\|_{L^p(\Omega)}$ is compactly embedded into $L^p(\Omega)$, the number $\lambda_p(\Omega)$ coincides with the smallest $\lambda\in\mathbb{R}$ such that the boundary value problem
\[
-\Delta_p u=\lambda\,|u|^{p-2}\,u,\quad \mbox{ in }\Omega,\qquad u=0,\quad \mbox{ on }\partial\Omega,
\]
does admit a nontrivial solution $u\in \mathcal{D}^{1,p}_0(\Omega)$. Here $\Delta_p$ is the quasilinear operator
\[
\Delta_p u=\mathrm{div\,}(|\nabla u|^{p-2}\,\nabla u),
\]
known as {\it $p-$Laplacian}.
For this reason, $\lambda_p(\Omega)$ is called {\it first eigenvalue of the $p-$Laplacian with Dirichlet conditions on} $\Omega$. In this case as well, we have the sharp lower bound
\[
\lambda_p(\Omega)\ge \left(|B|^\frac{p}{N}\,\lambda_p(B)\right)\,\frac{1}{|\Omega|^\frac{p}{N}},
\]
which generalizes \eqref{FK} to $p\not=2$. 
The main goal of this paper is to generalize the Hersch-Protter estimate \eqref{hersch} to the case of $\lambda_p$. At this aim, we introduce the one-dimensional Poincar\'e constant
\[
\pi_p=\inf_{\varphi\in C^1([0,1])\setminus\{0\}} \left\{\frac{\|\varphi'\|_{L^p([0,1])}}{\|\varphi\|_{L^p([0,1])}}\, :\, \varphi(0)=\varphi(1)=0\right\}.
\]
We will prove the following
\begin{theorem}
\label{teo:herschp}
Let $\Omega\subset\mathbb{R}^N$ be an open convex set. Then we have
\begin{equation}
\label{herschp}
\lambda_p(\Omega)\ge \left(\frac{\pi_p}{2}\right)^p\,\frac{1}{R_\Omega^p}.
\end{equation}
The estimate is sharp, equality being attained for example:
\begin{itemize}
\item by an infinite slab, i.e. a set of the form
\[
\Big\{x\in\mathbb{R}^N\, :\, a< \langle x,\omega\rangle< b\Big\},
\]
for some $a<b$ and $\omega\in\mathbb{S}^{N-1}$;
\vskip.2cm
\item asymptotically by the family of  ``collapsing pyramids'' 
\[
C_\alpha=\mathrm{convex\, hull}\Big((-1,1)^{N-1}\cup \{(0,\dots,0,\alpha)\}\Big),
\]
in the sense that
\[
\lim_{\alpha\to 0^+}R_{C_\alpha}^p\,\lambda_p(C_\alpha)=\left(\frac{\pi_p}{2}\right)^p;
\]
\item more generally, asymptotically by the family of infinite slabs with section given by a $k-$dimensional collapsing pyramid, i.e.
\[
\mathbb{R}^{N-k}\times C_\alpha,\qquad\mbox{ for } N\ge 3 \mbox{ and } 2\le k\le N-1.
\]
\end{itemize}
\end{theorem}

\begin{remark}[More general sets II]
For $p\not=2$, the case of more general sets has been investigated by Poliquin in \cite{Po}. In \cite[Theorem 1.4.1]{Po} it is proved that for $p>N$ and $\Omega\subset\mathbb{R}^N$ open bounded set, one has
\[
\lambda_p(\Omega)\ge \frac{C}{R_\Omega^p},
\] 
for a constant $C=C(N,p)>0$. Then in \cite[Theorem 1.4.2]{Po} the same estimate is proved, for $p>N-1$ and $\Omega$ having a connected boundary. In both cases, the constant $C$ is not explicit.
\end{remark}
As already observed by Makai in the case $p=N=2$ (see \cite{Ma}), the estimate of Theorem \ref{teo:herschp} in turn implies another interesting lower bound on $\lambda_p(\Omega)$, this time in terms of the quantity
\[
\frac{P(\Omega)}{|\Omega|},
\]
where $P(\Omega)$ is the perimeter of $\Omega$. The resulting estimate, which seems to be new for $N\ge 3$ and $p\not =2$, is contained in Corollary \ref{coro:herschp} below.

\begin{remark}[Upper bound]
Up to now, we never mentioned the possibility of having an upper bound of the type
\[
\lambda_p(\Omega)\le \frac{C}{R_\Omega^p}.
\]
The reason is simple: such an estimate is indeed true and very simple to obtain in a sharp form, without any assumption on the set $\Omega$. Indeed, by definition of $\lambda_p$ it is easy to see that this is a monotone decreasing quantity, with respect to set inclusion. Thus, if $\Omega\subset\mathbb{R}^N$ is an open set with $R_\Omega<+\infty$, there exists a ball $B_{R_\Omega}(\xi)\subset \Omega$ and we have
\[
\lambda_p(\Omega)\le \lambda_p(B_{R_\Omega}(\xi)).
\]
If we now use the scaling properties of $\lambda_p$, the previous can be rewritten as
\[
\lambda_p(\Omega)\le \frac{\lambda_p(B_{1}(0))}{R_\Omega^p}.
\]
Observe that this estimate is sharp, equality being (uniquely) attained by balls.
\end{remark}

\subsection{Plan of the paper}
In Section \ref{sec:2} we introduce the notation used throughout the whole paper and the technical facts needed to handle the proof of Theorem \ref{teo:herschp}. Section \ref{sec:3} contains a rougher version of our main result, based on Hardy's inequality for convex sets. This is a sort of {\it divertissement}, that we think to be interesting in its own. The proof of Theorem \ref{teo:herschp} is then contained in Section \ref{sec:4}. We combine this result with a geometric estimate, to obtain a further lower bound on $\lambda_p$ of geometric nature: this is Section \ref{sec:5}, which also contains a lower bound on the {\it Cheeger constant}. Finally, in the last Section \ref{sec:6} we consider the same type of lower bound in terms of the inradius, with $\lambda_p$ replaced by a general Poincar\'e-Sobolev sharp constant. The paper ends with an open problem.

\begin{ack} 
We thank Berardo Ruffini for some comments on a preliminary version of this paper and for pointing out the reference \cite{Po}.
This paper evolved from a set of hand-written notes for a talk delivered during the conferences ``{\it Variational and PDE problems in Geometric Analysis}'' and ``{\it Recent advances in Geometric Analysis}'' held in June 2018 in Bologna and Pisa, respectively. The organizers Chiara Guidi \& Vittorio Martino and Andrea Malchiodi \& Luciano Mari are kindly acknowledged.
\end{ack}

\section{Preliminaries}
\label{sec:2}

\subsection{Notation}
For an open set $\Omega\subset\mathbb{R}^N$, we indicate by $|\Omega|$ its $N-$dimensional Lebesgue measure. For an open bounded set $\Omega\subset\mathbb{R}^N$ with Lipschitz boundary, we define the {\it distance function}
\[
d_\Omega(x)=\inf_{y\in\partial\Omega} |x-y|,\qquad x\in\Omega.
\]
Then we recall that the inradius $R_\Omega$ of $\Omega$ coincides with
\[
R_\Omega=\sup_{x\in\Omega} d_\Omega(x).
\]
We will set $\nu_\Omega(x)$ to be the outer normal versor at $\partial\Omega$, whenever this is well-defined.
\begin{definition}
We say that $\Omega\subset\mathbb{R}^N$ is an {\it open polyhedral convex set} if there exists a finite number of open half-spaces $\mathcal{H}_1,\dots,\mathcal{H}_k\subset\mathbb{R}^N$
such that
\[
\Omega=\bigcap_{i=1}^k \mathcal{H}_i\not=\emptyset.
\]
If $\Omega$ is an open polyhedral convex set, we say that $F\subset\partial\Omega$ is a {\it face of $\Omega$} if the following hold:
\begin{itemize}
\item $F\not=\emptyset$;
\vskip.2cm
\item $F\subset \partial\mathcal{H}_i$, for some $i=1,\dots,k$;
\vskip.2cm
\item for any $E\subset\partial \Omega\cap \partial\mathcal{H}_i$ such that $F\subset E$, we have $E=F$.
\end{itemize}
\end{definition}
If $\Omega\subset\mathbb{R}^N$ is an open convex set with $R_\Omega<+\infty$, we know that there exists $\xi\in\Omega$ such that $B_{R_\Omega}(\xi)\subset\Omega$. Accordingly, we define the {\it contact set}
\[
\mathcal{C}_{\Omega,\xi}=\partial\Omega\cap \partial B_{R_\Omega}(\xi).
\]
Finally, we recall the definition
\[
\pi_p=\inf_{\varphi\in C^1([0,1])\setminus\{0\}} \left\{\frac{\|\varphi'\|_{L^p((0,1))}}{\|\varphi\|_{L^p((0,1))}}\, :\, \varphi(0)=\varphi(1)=0\right\}.
\]
It is not difficult to see that
\begin{equation}
\label{pi}
\pi_1=\pi_\infty=2\qquad \mbox{ and }\qquad \pi_2=\pi,
\end{equation}
see Lemmas \ref{lm:pi1} and \ref{lm:pinfty}.
\subsection{A geometric lemma}
The following geometric result is one of the building blocks of the proof of Theorem \ref{teo:herschp}. It is a higher-dimensional analogue of a simple two-dimensional fact used by Hersch in \cite{He}.
 This is the same as \cite[Lemmas 5.2 \& 5.3]{BuGuMa}, to which we refer for the proof.
\begin{lemma}
\label{lm:madonne}
Let $\Omega\subset\mathbb{R}^N$ be an open bounded convex set. Let $\xi\in\Omega$ be such that $B_{R_\Omega}(\xi)\subset\Omega$. Then there exists $m\ge 2$ and $\{P^1,\dots,P^m\}\subset \mathcal{C}_{\Omega,\xi}$ distinct points such that the open polyhedral convex domain
\[
T=\bigcap_{i=1}^m \{x\in\mathbb{R}^N\, :\, \langle x-P^i,\nu_\Omega(P^i)\rangle<0\},
\]
has the following properties:
\begin{itemize}
\item $\Omega\subset T$;
\vskip.2cm
\item $R_T=R_\Omega$;
\vskip.2cm
\item every face of $T$ touches $\partial B_{R_\Omega}(\xi)$.
\vskip.2cm
\end{itemize}
\end{lemma}
\begin{remark}
The previous result is similar to an analogous geometric lemma contained in Protter's paper, see \cite[page 68]{Pr}. Such a result in \cite{Pr} is credited to a private communication by David Gale, without giving a proof. It should be noticed that the statement in \cite{Pr} is slightly more precise, since it is said that $m$ can be chosen to be smaller than or equal to $N+1$.
However, in the statement contained \cite{Pr} the crucial feature that all the faces of $T$ touches the internal ball $B_{R_\Omega}(\xi)$ seems to have been accidentally omitted. For this reason we prefer to refer to the result proved in \cite{BuGuMa}.
\end{remark}

\subsection{Eigenvalues of special sets}
\begin{lemma}[Product sets]
\label{lm:product}
Let $1<p<+\infty$ and $k\in\{1,\dots,N-1\}$. We take the open set $\Omega=\mathbb{R}^{N-k}\times\omega$, with $\omega\subset \mathbb{R}^k$ open bounded set. Then we have
\[
\lambda_p(\Omega)=\lambda_p(\omega).
\]
\end{lemma}
\begin{proof}
The proof is standard, we include it for completeness. 
\par
We use the notation $(x,y)\in\mathbb{R}^{N-k}\times\mathbb{R}^k$, for a point in $\mathbb{R}^N$.
We first prove that
\begin{equation}
\label{sopraomega}
\lambda_p(\Omega)\le \lambda_p(\omega).
\end{equation}
For every $\varepsilon>0$, we take $u_\varepsilon\in C^\infty_0(\omega)$ to be an almost optimal function for the problem on $\omega$, i.e.
\[
\int_\omega |\nabla_y u_\varepsilon|^p\,dy<\lambda_p(\omega)+\varepsilon\qquad \mbox{ and }\qquad \int_\omega |u_\varepsilon|^p\,dy=1.
\]
We take $\eta\in C^\infty_0(\mathbb{R})$  such that
\[
0\le \eta\le 1,\qquad \eta\equiv 1 \mbox{ on } \left[-\frac{1}{2},\frac{1}{2}\right],\qquad \eta\equiv 0 \mbox{ on } \mathbb{R}\setminus[-1,1],
\]
then for every $R>0$, we choose 
\[
\varphi(x,y)=\eta_R(|x|)\,u_\varepsilon(y),\qquad \mbox{ where }\ \eta_R(t)=R^{\frac{k-N}{p}}\,\eta\left(\frac{t}{R}\right).
\]
By using Fubini's Theorem, we obtain
\[
\lambda_p(\Omega)\le \frac{\displaystyle \int_{B_{R}(0)}\int_\omega\left(\left|\nabla _{x}\eta_R\left(|x|\right)\right|^2\,|u_\varepsilon(x_N)|^2+|\nabla_y u_\varepsilon(y)|^2\,\eta_R(|x|)^2\right)^\frac{p}{2}\,dx\,dy}{\displaystyle \int_{B_R(0)} \eta_R(|x|)^p\,dx},
\]
where $B_R(0)=\{x\in\mathbb{R}^{N-k}\, :\, |x|<R\}$.
We now use the definition of $\eta_R$ and the change of variables $x=R\,x'$, so to get
\[
\begin{split}
\lambda_p(\Omega)&\le \frac{\displaystyle \int_{B_1(0)}\int_\omega \left[R^{\frac{2}{p}\,(k-N)-2}\,\left|\eta'(|x'|)\right|^2\,|u_\varepsilon(y)|^2+R^{\frac{2}{p}\,(k-N)}\,|\nabla_y u_\varepsilon(y)|^2\,|\eta(|x'|)|^2\right]^\frac{p}{2}R^{N-k}\,dx'\,dy}{\displaystyle\int_{B_1(0)} \eta(|x'|)^p\,dx'}\\
&=\frac{\displaystyle \int_{B_1(0)}\int_0^1 \left[\frac{1}{R^2}\,\left|\eta'(|x'|)\right|^2\,|u_\varepsilon(y)|^2+|\nabla_y u_\varepsilon(y)|^2\,|\eta(|x'|)|^2\right]^\frac{p}{2}\,dx'\,dy}{\displaystyle\int_{B_1(0)}\,\eta(|x'|)^p\,dx'}.
\end{split}
\]
By taking the limit as $R$ goes to $+\infty$ and using the Dominated Convergence Theorem, from the previous estimate we get
\[
\lambda_p(\Omega)\le \frac{\displaystyle \int_{B_1(0)}\int_\omega |\nabla_y u_\varepsilon(y)|^p\,|\eta(|x'|)|^p\,dx'\,dy}{\displaystyle\int_{B_1(0)}\,\eta(|x'|)^p\,dx'}=\int_0^1 |\nabla_y u_\varepsilon|^p\,dy<\lambda_p(\omega)+\varepsilon.
\]
The arbitrariness of $\varepsilon>0$ implies \eqref{sopraomega}.
\vskip.2cm\noindent
We now prove the reverse inequality 
\begin{equation}
\label{sottoomega}
\lambda_p(\Omega)\ge \lambda_p(\omega).
\end{equation}
For every $\varepsilon>0$, we take $\varphi_\varepsilon\in C^\infty_0(\Omega)\setminus\{0\}$ such that
\[
\frac{\displaystyle\int_\Omega |\nabla \varphi_\varepsilon|^p\,dx\,dy}{\displaystyle\int_\Omega |\varphi_\varepsilon|^p\,dx\,dy}<\lambda_p(\Omega)+\varepsilon.
\]
Observe that
\[
\begin{split}
\int_{\Omega} |\nabla \varphi_\varepsilon|^p\,dx\,dy&\ge \int_{\mathbb{R}^{N-k}}\left(\int_\omega |\nabla_y \varphi_\varepsilon|^p\,d y\right)\,dx\\
&\ge \lambda_p(\omega)\,\int_{\mathbb{R}^{N-k}}\left(\int_\omega |\varphi_\varepsilon|^p\,d y\right)\,dx=\lambda_p(\omega)\,\int_\Omega|\varphi_\varepsilon|^p\,dx\,dy,
\end{split}
\]
where we used that $y\mapsto \varphi_\varepsilon(x,y)$ is admissible for the one-dimensional problem, for every $x$. We thus obtained
\[
\lambda_p(\omega)\le \lambda_p(\Omega)+\varepsilon.
\]
The arbitrariness of $\varepsilon>0$ implies \eqref{sottoomega}.
\end{proof}
The following technical result is the core of the proof of Theorem \ref{teo:herschp}. It enables to estimate from below an eigenvalue with mixed boundary conditions, when the set is a ``pyramid-like'' one. We have to pay attention to possibly unbounded sets. In what follows $W^{1,p}(\Omega)$ is the usual Sobolev space of $L^p(\Omega)$ functions, having their distributional gradient in $L^p(\Omega)$, as well.
\begin{lemma}
\label{lm:gale}
Let $\Sigma\subset\mathbb{R}^{N-1}$ be an open polyhedral convex set. Let $\xi=(\xi_1,\dots,\xi_N)\in\mathbb{R}^N$ be a point whose projection on $\mathbb{R}^{N-1}$ belongs to $\Sigma$ and such that $\xi_N>0$. We consider the $N-$dimensional polyhedral convex set
\[
T=\mathrm{convex\, hull}\big(\Sigma\cup \{\xi\}\big),
\]
and define
\[
\mu(T)=\inf_{u\in C^1(\overline{T})\cap W^{1,p}(T)\setminus\{0\}} \left\{\frac{\displaystyle\int_T |\nabla u|^p\,dx}{\displaystyle\int_T |u|^p\,dx}\, :\, u=0 \mbox{ on } \Sigma\right\}. 
\]
Then we have
\[
\mu(T)\ge \left(\frac{\pi_p}{2}\right)^p\,\frac{1}{(\xi_N)^p}.
\]
\end{lemma}
\begin{proof}
By recalling the definition of $\pi_p$, we have that for $a>0$ and for every $\varphi\in C^1([0,a])$ such that $\varphi(0)=0$ it holds
\begin{equation}
\label{oned}
\int_{0}^a |\varphi'(t)|^p\ge \left(\frac{\pi_p}{2}\right)^p\,\frac{1}{a^p}\,\int_0^a |\varphi(t)|^p\,dt,
\end{equation}
see \cite[Lemma A.1]{B}.
We now take a function $u\in C^1(\overline{T})\cap W^{1,p}(T)$ which is admissible for the problem defining $\mu(T)$. By hypothesis, there exists an affine function $\Psi:\Sigma\to [0,\xi_N]$ such that
\[
T=\left\{(x',x_N)\, :\, \mathbb{R}^{N-1}\times\mathbb{R}\, :\, x'\in \Sigma,\, 0<x_N<\Psi(x')\right\}.
\]
Thus by Fubini's Theorem and \eqref{oned} we have
\[
\begin{split}
\int_T |\nabla u|^p\,dx&\ge \int_T |u_{x_N}|^p\,dx= \int_\Sigma\left(\int_{0}^{\Psi(x')} |u_{x_N}|^p\,dx_N\right)\,dx'\\
&\ge \int_\Sigma\left( \left(\frac{\pi_p}{2}\right)^p\,\frac{1}{\Psi(x')^p}\,\int_{0}^{\Psi(x')} |u|^p\,dx_N\right)\,dx'\\
&\ge \left(\frac{\pi_p}{2}\right)^p\,\frac{1}{\xi_N^p}\,\int_\Sigma\left( \int_{0}^{\Psi(x')} |u|^p\,dx_N\right)\,dx'=\left(\frac{\pi_p}{2}\right)^p\,\frac{1}{\xi_N^p}\,\int_T |u|^p\,dx.
\end{split}
\]
By taking the infimum over admissible functions $u$, we get the desired conclusion.
\end{proof}

\section{A {\it divertissement} on Hardy's inequality}
\label{sec:3}

Before proving the sharp estimate {\it \`a la} Hersch-Protter \eqref{herschp}, we present a rougher estimate. This is a consequence of {\it Hardy's inequality for convex sets}. Even if the resulting estimate is not sharp, we believe that the proof has its own interest and we reproduce it for the reader's convenience.
\begin{proposition}
\label{prop:hardy}
Let $1<p<\infty$ and let $\Omega\subset\mathbb{R}^N$ be an open bounded convex set. Then we have
\[
\left(\frac{p-1}{p}\right)^p\,\frac{1}{R_\Omega^p}\le \lambda_p(\Omega).
\]
\end{proposition}
\begin{proof}
We recall that the following Hardy's inequality holds for a convex set
\begin{equation}
\label{hardy}
\left(\frac{p-1}{p}\right)^p\,\int_\Omega \left|\frac{u}{d_\Omega}\right|^p\,dx\le \int_\Omega |\nabla u|^p\,dx,\qquad \mbox{ for every } u\in C^\infty_0(\Omega).
\end{equation}
By using this inequality, it is easy to obtain the claimed estimate.
By recalling that 
\[
R_\Omega=\|d_\Omega\|_{L^\infty(\Omega)},
\]
from \eqref{hardy} we get
\[
\left(\frac{p-1}{p}\right)^p\,\frac{1}{R_\Omega^p}\,\int_\Omega |u|^p\,dx<\int_\Omega |\nabla u|^p\,dx.
\]
By taking the infimum over admissible test functions, we finally obtain the lower bound on $\lambda_p(\Omega)$.
\par
For completeness, we now recall how to prove \eqref{hardy}. Let us consider the distance function
\[
d_\Omega(x)=\min_{y\in\partial \Omega} |x-y|,\qquad x\in\Omega.
\] 
This is a $1-$Lipschitz function, which is concave on $\Omega$, due to the convexity of $\Omega$. This implies that $d_\Omega$ is weakly superharmonic, i.e.
\[
\int_\Omega \langle \nabla d_\Omega,\nabla \varphi\rangle\,dx\ge 0,
\]
for every nonnegative $\varphi\in C^\infty_0(\Omega)$. By observing that 
\begin{equation}
\label{eikonal}
|\nabla d_\Omega|=1,\qquad \mbox{ almost everywhere in }\Omega,
\end{equation}
from the previous inequality we also get
\begin{equation}
\label{pweak}
\int_\Omega \langle |\nabla d_\Omega|^{p-2}\,\nabla d_\Omega,\nabla \varphi\rangle\,dx\ge 0,
\end{equation}
for every nonnegative $\varphi\in C^\infty_0(\Omega)$, i.e. $d_\Omega$ is weakly $p-$superharmonic as well. By a standard density argument, we easily see that we can enlarge the class of test functions up to  $\varphi\in W^{1,p}_0(\Omega)$, i.e. the closure of $C^\infty_0(\Omega)$ in $W^{1,p}(\Omega)$.
\par
We now insert in \eqref{pweak} the test function
\[
\varphi=\frac{|u|^p}{(d_\Omega+\varepsilon)^{p-1}},
\]
where $u\in C^\infty_0(\Omega)$ and $\varepsilon>0$. We thus obtain
\[
0\le -(p-1)\,\int_\Omega \left|\frac{\nabla d_\Omega}{d_\Omega+\varepsilon}\right|^p\,|u|^p\,dx+p\,\int_\Omega \left\langle \frac{|\nabla d_\Omega|^{p-2}\,\nabla d_\Omega}{(d_\Omega+\varepsilon)^{p-1}}, \nabla u\right\rangle\,|u|^{p-2}\,u\,dx.
\]
that is
\[
\int_\Omega \left|\frac{\nabla d_\Omega}{d_\Omega+\varepsilon}\right|^p\,|u|^p\,dx\le \frac{p}{p-1}\,\int_\Omega \left|\left\langle \frac{|\nabla d_\Omega|^{p-2}\,\nabla d_\Omega}{(d_\Omega+\varepsilon)^{p-1}}, \nabla u\right\rangle\right|\,|u|^{p-1}\,dx.
\]
We can now use Young's inequality in the following form
\[
|\langle a,b\rangle|\le \delta\,\frac{p-1}{p}\,|a|^\frac{p}{p-1}+\frac{\delta^{1-p}}{p}\,|b|^p,\qquad \mbox{ for } a,b\in\mathbb{R}^N,\, \delta>0.
\]
This yields
\[
\int_\Omega \left|\frac{\nabla d_\Omega}{d_\Omega+\varepsilon}\right|^p\,|u|^p\,dx\le \delta\,\int_\Omega \left|\frac{\nabla d_\Omega}{d_\Omega+\varepsilon}\right|^p\,|u|^p\,dx+\frac{\delta^{1-p}}{p-1}\,\int_\Omega |\nabla u|^p\,dx,
\]
which can be recast into
\[
(p-1)\,\delta^{p-1}\,(1-\delta)\,\int_\Omega \left|\frac{\nabla d_\Omega}{d_\Omega+\varepsilon}\right|^p\,|u|^p\,dx\le \int_\Omega |\nabla u|^p\,dx.
\]
Finally, we observe that the quantity $\delta^{p-1}\,(1-\delta)$ is maximal for 
\[
\delta=\frac{p-1}{p},
\]
thus by taking the limit as $\varepsilon$ goes to $0$ and recalling \eqref{eikonal}, by Fatou's Lemma we end up with \eqref{hardy}, as desired.
\end{proof}
\begin{remark}
We observe that the boundedness of $\Omega$ can be dropped, both in \eqref{hardy} and in the lower bound on $\lambda_p(\Omega)$.
We also point out that, even if the constant
\[
\left(\frac{p-1}{p}\right)^p,
\]
is not sharp, it only depends on $p$, just like the sharp one.
\end{remark}

\section{Proof of Theorem \ref{teo:herschp}}
\label{sec:4}
We start with a particular case of Theorem \ref{teo:herschp}, when the convex set is {\it polyhedral}. Its proof heavily relies on Lemma \ref{lm:gale}.
\begin{proposition}
\label{prop:fundamental}
Let $1<p<+\infty$ and let $T\subset\mathbb{R}^N$ be an open polyhedral convex set. We suppose that $R_T<+\infty$ and we assume further that there exists a ball $B\subset T$ with radius $R_T$ and such that each face of $T$ touches $B$. Then we have
\[
\lambda_p(T)\ge \left(\frac{\pi_p}{2}\right)^p\,\frac{1}{R_T^p}.
\]
\end{proposition}
\begin{proof}
Let us indicate by $F_1,\dots,F_j\subset\partial T$ the faces of $T$. We take the center $\xi$ of $B$ and then define 
\[
T_i=\mathrm{convex\, hull}\big(F_i\cup \{\xi\}\big),\qquad i=1,\dots,j,
\]
see Figures \ref{fig:2} and \ref{fig:3}.
\begin{figure}
\includegraphics[scale=.3]{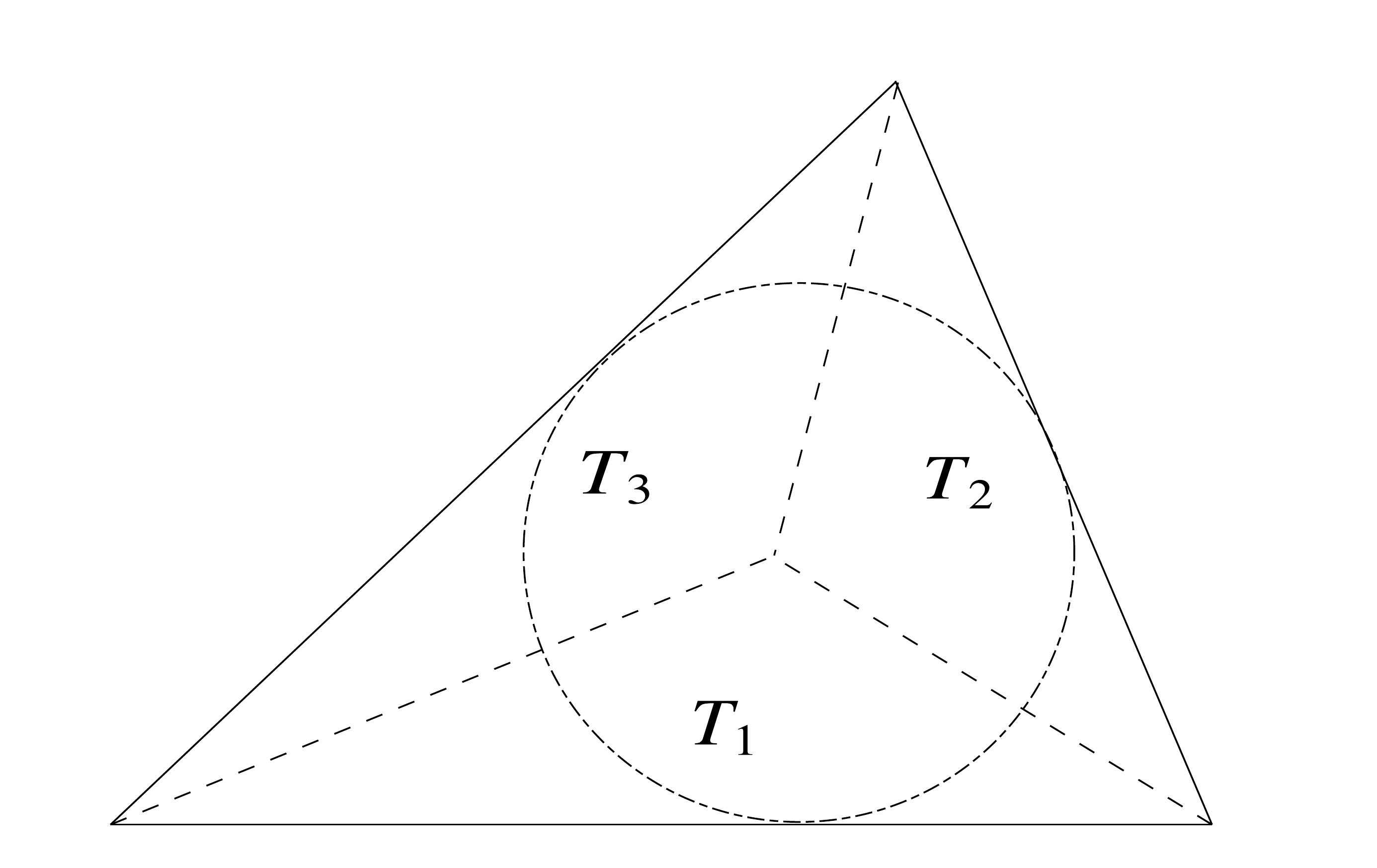}
\caption{The construction for the proof of Proposition \ref{prop:fundamental}, when $N=2$ and $T$ has $j=3$ faces.}
\label{fig:2}
\end{figure}
\begin{figure}
\includegraphics[scale=.3]{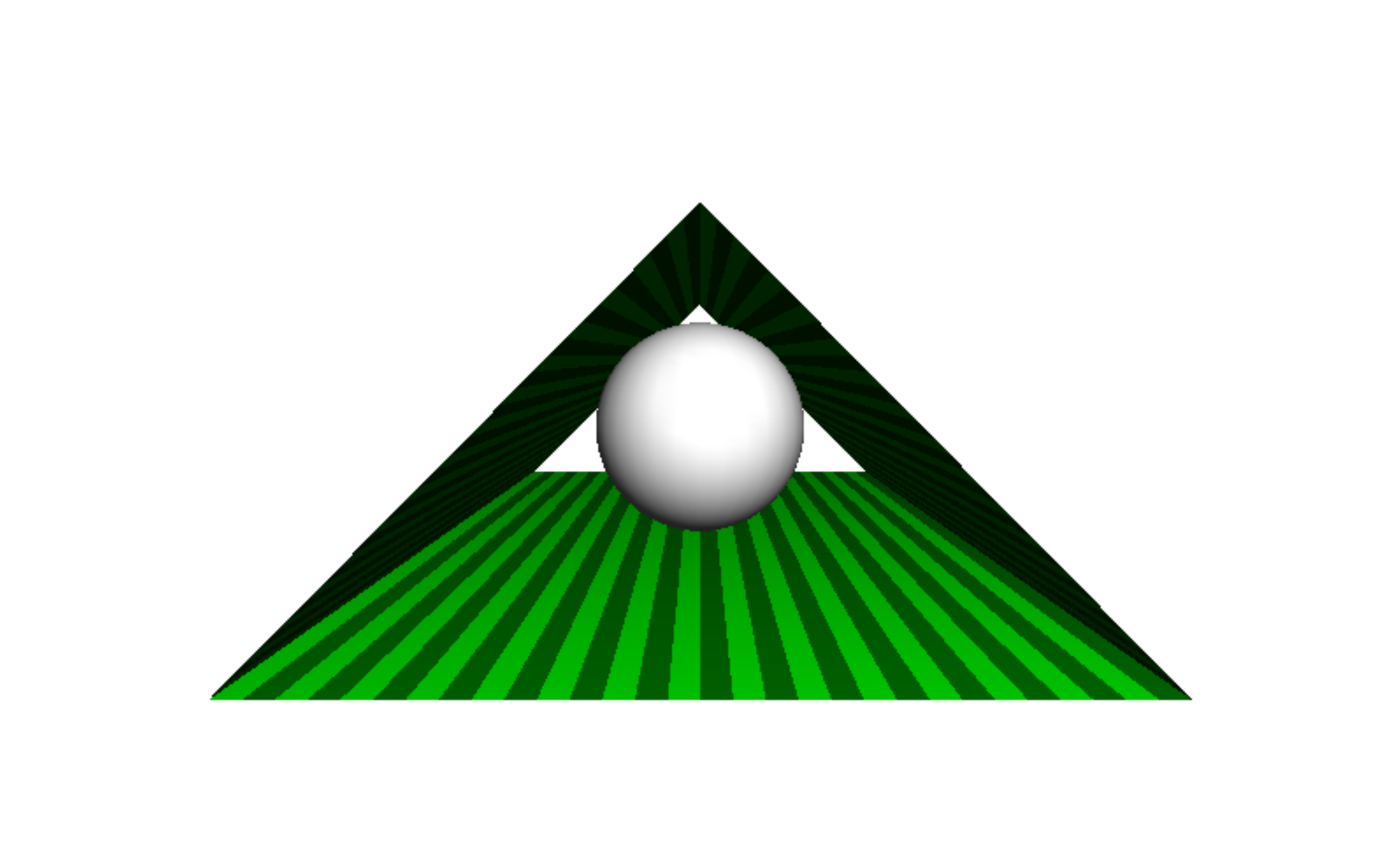}
\caption{The construction for the proof of Proposition \ref{prop:fundamental}, when $N=3$ and $T$ is an unbounded set with $j=3$ faces. In this case, the subsets $T_1,T_2,T_3$ (not drawn in the picture) are unbounded, as well.}
\label{fig:3}
\end{figure}
We now consider $T_i$ for a fixed $i=1,\dots,j$ and estimate from below
\[
\mu_i=\inf_{u\in C^1(\overline{T_i})\cap W^{1,p}(T_i)\setminus\{0\}} \left\{\frac{\displaystyle\int_{T_i} |\nabla u|^p\,dx}{\displaystyle\int_{T_i} |u|^p\,dx}\, :\, u=0 \mbox{ on } F_i\right\}.
\]
Up to a rigid motion, we can assume that $T_i$ satisfies the assumptions of Lemma \ref{lm:gale}. Observe that in this case, we have
\[
\xi_N=R_T,
\] 
by construction.
Thus we get
\begin{equation}
\label{egole}
\mu_i\ge \left(\frac{\pi_p}{2}\right)^p\,\frac{1}{(\xi_N)^p}= \left(\frac{\pi_p}{2}\right)^p\,\frac{1}{R_T^p}.
\end{equation}
On the other hand, for every $\varepsilon>0$, we take $\varphi_\varepsilon\in C^\infty_0(T)\setminus\{0\}$ such that
\[
\lambda_p(T)\le \frac{\displaystyle\int_T |\nabla \varphi_\varepsilon|^p\,dx}{\displaystyle \int_T |\varphi_\varepsilon|^p\,dx}\le \lambda_p(T)+\varepsilon.
\] 
We observe that the restriction of $\varphi_\varepsilon$ to each $T_i$ is admissible for the problem defining $\mu_i$.
Then, we obtain
\[
\begin{split}
\lambda_p(T)+\varepsilon&\ge \frac{\displaystyle\int_T |\nabla \varphi_\varepsilon|^p\,dx}{\displaystyle \int_T |\varphi_\varepsilon|^p\,dx}=\frac{\displaystyle\sum_{i=1}^j\int_{T_i} |\nabla \varphi_\varepsilon|^p\,dx}{\displaystyle \sum_{i=1}^j\int_{T_i} |\varphi_\varepsilon|^p\,dx}\ge\frac{\displaystyle\sum_{i=1}^j \mu_i\,\int_{T_i} |\varphi_\varepsilon|^p\,dx}{\displaystyle \sum_{i=1}^j\int_T |\varphi_\varepsilon|^p\,dx} \ge\min_{i=1,\dots,j}\ \mu_i.
\end{split}
\]
By recalling the lower bound \eqref{egole}, we get the the desired conclusion, thanks to the arbitrariness of $\varepsilon>0$.
\end{proof}
We eventually come to the proof of Theorem \ref{teo:herschp}.
\begin{proof}[Proof of Theorem \ref{teo:herschp}]
We first prove the inequality and then analyze the equality cases.
\vskip.2cm\noindent
{\bf Part 1: proof of the inequality.}  Let us first assume that $\Omega$ is bounded. By appealing to Lemma \ref{lm:madonne}, we know that there exists $T\subset \mathbb{R}^N$ an open polyhedral convex set such that 
\[
\Omega\subset T\qquad \mbox{ and }\qquad R_\Omega=R_T.
\]
Moreover, each face of $T$ touches a maximal ball $B_{R_\Omega}(\xi)$.  By applying Proposition \ref{prop:fundamental} to the set $T$, we get
\[
\lambda_p(\Omega)\ge \lambda_p(T)\ge \left(\frac{\pi_p}{2}\right)^p\,\frac{1}{R_T^p}=\left(\frac{\pi_p}{2}\right)^p\,\frac{1}{R_\Omega^p}.
\]
This concludes the proof, in the case $\Omega$ is bounded. 
\par
If $\Omega$ in unbounded, we can suppose that $R_\Omega<+\infty$, otherwise there is nothing to prove. Then we can consider the bounded set $\Omega_R=\Omega\cap B_R(0)$ for $R$ large enough. By applying
\[
\lambda_p(\Omega_R)\ge \left(\frac{\pi_p}{2}\right)^p\,\frac{1}{R_{\Omega_R}^p},
\]  
and taking on both sides the limit as $R$ goes to $+\infty$, we get the conclusion.
\vskip.2cm\noindent
{\bf Part 2: sharpness of the inequality.} It is easy to see that equality is attained on a slab. Indeed, by Lemma \ref{lm:product} we have
\[
\lambda_p(\mathbb{R}^{N-1}\times (0,1))=\lambda_p((0,1))=\Big(\pi_p\Big)^p\qquad \mbox{ and }\qquad R_{\mathbb{R}^{N-1}\times (0,1)}=\frac{1}{2}.
\]
As for the ``collapsing pyramids''
\[
C_\alpha=\mathrm{convex\, hull}\Big((-1,1)^{N-1}\cup \{(0,\dots,0,\alpha)\}\Big),
\]
we are going to use a purely variational argument, thus we not need the explicit determination of $\lambda_p$ for these sets. We first observe that 
\[
C_\alpha\subset\mathbb{R}^{N-1} \times(0,\alpha),
\]
thus we have
\[
\lambda_p(C_\alpha)\ge \lambda_p(\mathbb{R}^{N-1} \times(0,\alpha))=\left(\frac{\pi_p}{\alpha}\right)^p.
\]
In order to prove the reverse estimate, we observe that for $0<\alpha<1$
\[
Q_\alpha:=\Big(-(1-\sqrt{\alpha}),1-\sqrt{\alpha}\Big)^{N-1}\times \Big(0,\alpha\,(1-\sqrt{\alpha})\Big)\subset C_\alpha,
\]
thus by monotonicity and scaling
\[
\lambda_q(C_\alpha)\le \lambda_p(Q_\alpha)=\Big(\alpha\,(1-\sqrt{\alpha})\Big)^{-p}\,\lambda_p\left(\left(-\frac{1}{\alpha},\frac{1}{\alpha}\right)^{N-1}\times (0,1)\right).
\]
By observing that 
\[
\lim_{\alpha\to 0^+}\lambda_p\left(\left(-\frac{1}{\alpha},\frac{1}{\alpha}\right)^{N-1}\times (0,1)\right)=\lambda_p(\mathbb{R}^{N-1}\times(0,1))=\Big(\pi_p\Big)^p,
\]
we thus get that
\[
\lambda_p(Q_\alpha)\sim \left(\frac{\pi}{\alpha}\right)^p,\qquad \mbox{ for }\alpha\to 0^+.
\]
In conclusion, we obtained that 
\[
\lim_{\alpha\to 0^+} \alpha^p\,\lambda_p(C_\alpha)=\Big(\pi_p\Big)^p.
\]
We are left with observing that 
\[
R_{C_\alpha}=\frac{\alpha}{1+\sqrt{1+\alpha^2}}\sim \frac{\alpha}{2},\qquad \mbox{ for } \alpha \to 0^+.
\]
This concludes the proof of the optimality of the sequence $\{C_\alpha\}_\alpha$.
\par
Finally, we observe that for the sets 
\[
\mathbb{R}^{N-k}\times C_\alpha,\qquad\mbox{ for } N\ge 3 \mbox{ and } 2\le k\le N-1,
\]
it is sufficient to use the computations above and the fact that by Lemma \ref{lm:product}
\[
\lambda_p(\mathbb{R}^{N-k}\times C_\alpha)=\lambda_p(C_\alpha),
\]
together with $R_{\mathbb{R}^{N-k}\times C_\alpha}=R_{C_\alpha}$.
\end{proof}
\begin{remark}
By comparing the sharp estimate \eqref{hersch} with the estimate of Proposition \ref{prop:hardy}, we get
\[
\frac{\pi_p}{2}>\frac{p-1}{p}.
\]
By recalling \eqref{pi}, we have that both sides converge to $1$, as $p$ goes to $+\infty$. This shows that even if the estimate of Proposition \ref{prop:hardy} is not sharp for every finite $p$, it is ``asymptotically'' optimal for $p\to +\infty$.
\end{remark}

\section{A further lower bound}
\label{sec:5}

It what follows, we will use the notation $P(\Omega)$ to denote the distributional perimeter of a set $\Omega\subset\mathbb{R}^N$. On convex sets, this coincides with the $(N-1)-$dimensional Hausdorff measure of the boundary.
\par
We recall that for bounded convex sets, it is possible to bound $\lambda_p(\Omega)$ {\it from above} in terms of the isoperimetric--type ratio 
\[
\frac{P(\Omega)}{|\Omega|}.
\]
Namely, we have
\[
\lambda_p(\Omega)<\left(\frac{\pi_p}{2}\right)^p\,\left(\frac{P(\Omega)}{|\Omega|}\right)^p,
\]
see \cite[Main Theorem]{B} and \cite[Theorem 4.1]{DPG}. The inequality is strict and the estimate is sharp.
\vskip.2cm\noindent
As a straightforward consequence of Theorem \ref{teo:herschp}, we get that the previous estimate can be reverted. Thus
\[
\lambda_p(\Omega) \qquad \mbox{ and }\qquad \left(\frac{P(\Omega)}{|\Omega|}\right)^p,
\]
are equivalent quantities on open bounded convex sets. For $N=p=2$, this result is due to Makai, see \cite{Ma}. For all the other cases, to the best of our knowledge it is new.
\begin{corollary}
\label{coro:herschp}
Let $1<p<+\infty$ and let $\Omega\subset\mathbb{R}^N$ be an open bounded convex set. Then we have
\begin{equation}
\label{hc}
\lambda_p(\Omega)\ge \left(\frac{\pi_p}{2\,N}\right)^p\,\left(\frac{P(\Omega)}{|\Omega|}\right)^p.
\end{equation}
The inequality is sharp, equality being attained asymptotically by the sequence of ``collapsing pyramids'' of Theorem \ref{teo:herschp}. 
\end{corollary}
\begin{proof}
In order to prove \eqref{hc}, it is sufficient to recall that for an open bounded convex set, we have the sharp estimate (see for example \cite[Lemma A.1]{B})
\begin{equation}
\label{geometric}
\frac{R_\Omega}{N}\le \frac{|\Omega|}{P(\Omega)}.
\end{equation}
By inserting this in \eqref{herschp}, we get the claimed estimate.
\par
We now come to the sharpness issue. Observe that \eqref{hc} has been obtained by joining the two inequalities \eqref{herschp} and \eqref{geometric}. We already know that the family of ``collapsing pyramids'' is asymptotically optimal for the first one, thus we only need to verify that the same family is asymptotically optimal for \eqref{geometric}, as well. 
Let us set as before 
\[
C_\alpha=\mathrm{convex\, hull}\left(\Big(-1,1\Big)^{N-1}\cup\{(0,,\dots,0,\alpha)\}\right).
\]
We recall that
\[
R_{C_\alpha}\sim \frac{\alpha}{2},
\]
while 
\[
|C_\alpha|=2^{N-1}\,\int_0^\alpha \left(1-\frac{z}{\alpha}\right)^{N-1}\,dz=\frac{\alpha\,2^{N-1}}{N},
\]
and
\[
P(C_\alpha)\sim 2\,\left|\Big(-1,1\Big)^{N-1}\right|=2^N.
\]
Thus we get
\[
\frac{|C_\alpha|}{P(C_\alpha)}\sim \frac{\alpha}{2\,N}\sim \frac{R_{C_\alpha}}{N},\qquad \mbox{ for } \alpha\to 0,
\]
as desired.
\end{proof}
We recall the definition of {\it Cheeger constant} of an open bounded set $\Omega\subset\mathbb{R}^N$, i.e.
\[
h_1(\Omega)=\inf_{E\subset \Omega}\left\{\frac{P(E)}{|E|}\, :\, |E|>0\right\}.
\]
Observe that if $P(\Omega)<+\infty$, then $\Omega$ itself is admissible in the previous variational problem. Thus we have the trivial estimate
\[
\frac{P(\Omega)}{|\Omega|}\ge h_1(\Omega).
\]
For convex sets, this estimate can be reverted. Indeed, by recalling that (see \cite[Corollary 6]{KF})
\[
\lim_{p\searrow 1}\lambda_p(\Omega)=h_1(\Omega)\qquad \mbox{ and }\qquad \lim_{p\searrow 1}\pi_p=\pi_1=2,
\]
if we take the limit as $p$ goes to $1$ in \eqref{hc}, we get the following
\begin{corollary}
Let $1<p<+\infty$ and let $\Omega\subset\mathbb{R}^N$ be an open bounded convex set. Then we have
\[
h_1(\Omega)\ge \frac{1}{N}\,\frac{P(\Omega)}{|\Omega|}.
\]
\end{corollary}
\begin{remark}[The case $p=+\infty$]
The limit as $p$ goes to $+\infty$ of \eqref{hc} is less interesting. Indeed, by taking the $p-$th root on both sides and recalling that (see \cite[Lemma 1.5]{JLM})
\[
\lim_{p\to +\infty} \Big(\lambda_p(\Omega)\Big)^\frac{1}{p}=\frac{1}{R_\Omega},
\]
from \eqref{hc} we get again \eqref{geometric}.
\end{remark}

\section{More general principal frequencies}

\label{sec:6}

By appealing to its variational characterization, the first eigenvalue $\lambda_p(\Omega)$ is nothing but the sharp constant for the Poincar\'e inequality
\[
C_\Omega\,\int_\Omega |u|^p\,dx\le \int_\Omega |\nabla u|^p\,dx,\qquad  \mbox{ for every } u\in C^\infty_0(\Omega).
\]
From a theoretical point of view, it is thus quite natural to consider more generally the ``principal frequencies''
\[
\lambda_{p,q}(\Omega)=\inf_{u\in C^\infty_0(\Omega)\setminus \{0\}} \frac{\displaystyle \int_\Omega |\nabla u|^p\,dx}{\displaystyle \left(\int_\Omega |u|^q\,dx\right)^\frac{p}{q}},\qquad \mbox{ for } q\not =p. 
\]
Of course, such a quantity is interesting only if $q$ is such that
\[
\left\{\begin{array}{ll}
1\le q<p^*,& \mbox{ if } p\le N,\\
1\le q\le +\infty, & \mbox{ if } p>N,
\end{array}
\right.\qquad \mbox{ where } p^*=\frac{N\,p}{N-p}.
\]
For $p<N$ and $q=p^*$, the quantity $\lambda_{p,q}(\Omega)$ does not depend on $\Omega$ and is a universal constant, coinciding with the sharp constant in the Sobolev inequality
\[
C\,\left(\int_{\mathbb{R}^N} |u|^{p^*}\,dx\right)^\frac{p}{p^*}\le \int_{\mathbb{R}^N} |\nabla u|^p\,dx,\qquad  \mbox{ for every } u\in C^\infty_0(\mathbb{R}^N).
\]
In this section, we briefly investigate the possibility to have a lower bound of the type 
\[
\frac{C}{R_\Omega^\beta}\le \lambda_{p,q}(\Omega),
\]
among convex sets, in this case as well. Observe that by scale invariance, the only possibility for the exponent $\beta$ is 
\[
\beta=N-p-N\,\frac{p}{q}.
\]
In the case $q<p$, such an estimate is not possible, as shown in the following
\begin{proposition}[Sub-homogeneous case]
Let $1<p<+\infty$ and $1\le q<p$. Then 
\[
\inf \Big\{R_\Omega^{N\,\frac{p}{q}-N+p}\,\lambda_{p,q}(\Omega)\, :\, \Omega\subset\mathbb{R}^N \mbox{open bounded convex set}\Big\}=0.
\]
\end{proposition}
\begin{proof}
By scale invariance, we can impose the further restriction that $R_\Omega=1$.
We recall that for $q<p$ we have
\[
\lambda_{p,q}(\Omega)>0\qquad \Longleftrightarrow\qquad \mbox{ the embedding }\, \mathcal{D}^{1,p}_0(\Omega)\hookrightarrow L^q(\Omega)\, \mbox{ is compact},
\]
see \cite[Theorem 1.2]{BR}.
We now observe that for the open convex set $\Omega=\mathbb{R}^{N-1}\times (-1,1)$ the embedding above can not be compact, due to the translation invariance of the set $\Omega$ in the first $N-1$ coordinate directions. Thus we get
\[
\lambda_{p,q}(\mathbb{R}^{N-1}\times(-1,1))=0.
\]
By taking the sequence 
\[
\Omega_L=\left(-\frac{L}{2},\frac{L}{2}\right)^{N-1}\times(-1,1),\qquad L>0,
\]
and using that
\[
\lim_{L\to+\infty}\lambda_{p,q}(\Omega_L)=\lambda_{p,q}(\mathbb{R}^{N-1}\times(-1,1)),
\]
we get the desired conclusion.
\end{proof}
\begin{remark}[Torsional rigidity]
For $q=1$, the quantity 
\[
T_p(\Omega)=\frac{1}{\lambda_{p,1}(\Omega)},
\]
is usually called {\it $p-$torsional rigidity}. The previous results shows that an estimate of the form
\[
T_p(\Omega)\le C\, R_\Omega^{N\,p+p-N},
\]
is not possible.
\end{remark}

On the contrary, for $q>p$ it is possible to have a lower bound on $\lambda_{p,q}$ in terms of the inradius.
\begin{proposition}[Super-homogeneous case]
Let $1<p<\infty$ and $q>p$ such that
\[
\left\{\begin{array}{ll}
q<p^*,& \mbox{ if } p\le N,\\
q\le +\infty, & \mbox{ if } p>N.
\end{array}
\right.
\]
Then there exists a constant $C=C(N,p,q)>0$ such that for every $\Omega\subset\mathbb{R}^N$ open convex set, we have
\[
\lambda_{p,q}(\Omega)\ge \frac{C}{R_\Omega^{N\,\frac{p}{q}-N+p}}.
\]
\end{proposition}
\begin{proof}
By using the classical Gagliardo-Nirenberg inequalities, we have for every $u\in C^\infty_0(\Omega)$
\begin{equation}
\label{GNS}
\left(\int_\Omega |u|^q\,dx\right)^\frac{p}{q}\le C\, \left(\int_\Omega |u|^p\,dx\right)^\vartheta\,\left(\int_\Omega |\nabla u|^p\,dx\right)^{1-\vartheta},
\end{equation}
where $C=C(N,p,q)>0$ and 
\[
\vartheta=\frac{N}{q}-\frac{N}{p}+1.
\]
For every $\varepsilon>0$, we take $\varphi\in C^\infty_0(\Omega)$ such that
\[
\lambda_{p,q}(\Omega)+\varepsilon>\frac{\displaystyle\int_\Omega |\nabla \varphi|^p\,dx}{\left(\displaystyle\int_\Omega |\varphi|^q\,dx\right)^\frac{p}{q}}.
\]
By using \eqref{GNS} to estimate the denominator, we end up with
\[
\lambda_{p,q}(\Omega)+\varepsilon>\left(\frac{\displaystyle\int_\Omega |\nabla \varphi|^p\,dx}{\displaystyle\int_\Omega |\varphi|^p\,dx}\right)^\vartheta\ge \Big(\lambda_p(\Omega)\Big)^\vartheta.
\]
If we now use Theorem \ref{teo:herschp} and recall the definition of $\vartheta$, we get the desired conclusion.
\end{proof}
The previous proof very likely does not produce the sharp constant. On the other hand, the Hersch's argument used for the case $p=q$ does not seem to work in this case. Thus, we leave an open problem, which is quite interesting in our opinion.
\begin{open}
Find the sharp constant $C=C(N,p,q)>0$ such that for $p<q$
\[
\lambda_{p,q}(\Omega)\ge \frac{C}{R_\Omega^{N\,\frac{p}{q}-N+p}},\qquad \mbox{ for every } \Omega\subset\mathbb{R}^N \mbox{ open bounded convex set}.
\]
\end{open}

\appendix

\section{$\pi_1$ and $\pi_\infty$}

We observed in Section \ref{sec:2} that
\[
\pi_1=\pi_\infty=2.
\]
For the reader's convenience, we present a proof of these facts.
\begin{lemma}
\label{lm:pi1}
We have
\[
\pi_1=\inf_{\varphi\in C^1([0,1])\setminus\{0\}} \left\{\frac{\displaystyle\int_0^1 |\varphi'|\,dt}{\displaystyle\int_0^1 |\varphi|\,dt}\, :\, \varphi(0)=\varphi(1)=0\right\}=2.
\]
\end{lemma}
\begin{proof}
We take an admissible test function $\varphi$, for every $t\in[0,1/2]$ we have
\[
|\varphi(t)|=|\varphi(t)-\varphi(0)|=\left|\int_0^t \varphi'(\tau)\,d\tau\right|\le \int_0^t |\varphi'(\tau)|\,d\tau.
\]
By integrating over $[0,1/2]$ and exchanging the order of integration, we obtain
\begin{equation}
\label{1}
\begin{split}
\int_0^\frac{1}{2} |\varphi(t)|\,dt\le \int_0^\frac{1}{2}\left(\int_0^t |\varphi'(\tau)|\,d\tau\right)\,dt&=\int_0^\frac{1}{2}|\varphi'(\tau)|\,\left(\int_\tau^\frac{1}{2}dt\right)\,d\tau\\
&=\int_0^\frac{1}{2} \left(\frac{1}{2}-\tau\right)\,|\varphi'(\tau)|\,d\tau.
\end{split}
\end{equation}
Similarly, for every $t\in[1/2,1]$ we have
\[
|\varphi(t)|=|\varphi(1)-\varphi(t)|=\left|\int_t^1 \varphi'(\tau)\,d\tau\right|\le \int_t^1 |\varphi'(\tau)|\,d\tau.
\]
By integrating over $[1/2,1]$ and exchanging again the order of integration,  we obtain
\begin{equation}
\label{2}
\begin{split}
\int_\frac{1}{2}^1 |\varphi(t)|\,dt\le \int_\frac{1}{2}^1\left(\int_t^1 |\varphi'(\tau)|\,d\tau\right)\,dt&=\int_\frac{1}{2}^1|\varphi'(\tau)|\,\left(\int_\frac{1}{2}^\tau dt\right)\,d\tau\\
&=\int_\frac{1}{2}^1 \left(\tau-\frac{1}{2}\right)\,|\varphi'(\tau)|\,d\tau.
\end{split}
\end{equation}
If we now sum \eqref{1} and \eqref{2}, we get
\[
\int_0^1 |\varphi(t)|\,dt\le \int_0^1 \left|\frac{1}{2}-\tau\right| |\varphi'(\tau)|\,d\tau\le \frac{1}{2}\,\int_0^1 |\varphi'(\tau)|\,d\tau.
\]
This proves that $\pi_1\ge 2$.
\par
In order to prove the reverse estimate, we fix $0<\delta<1/2$ and take the piecewise affine function
\[
\varphi_\delta(t)=\left\{\begin{array}{rl}
0,& \mbox{ if } 0\le t<\delta,\\
\displaystyle\frac{t-\delta}{\delta}, & \mbox{ if } \delta<t<2\,\delta,\\
1,& \mbox{ if } 2\,\delta\le t\le 1-2\,\delta,\\
\displaystyle\frac{1-\delta-t}{\delta},& \mbox{ if } 1-2\,\delta<t<1-\delta,\\
0,& \mbox{ if } 1-\delta\le t\le 1.
\end{array}
\right.
\]
\begin{figure}
\includegraphics[scale=.3]{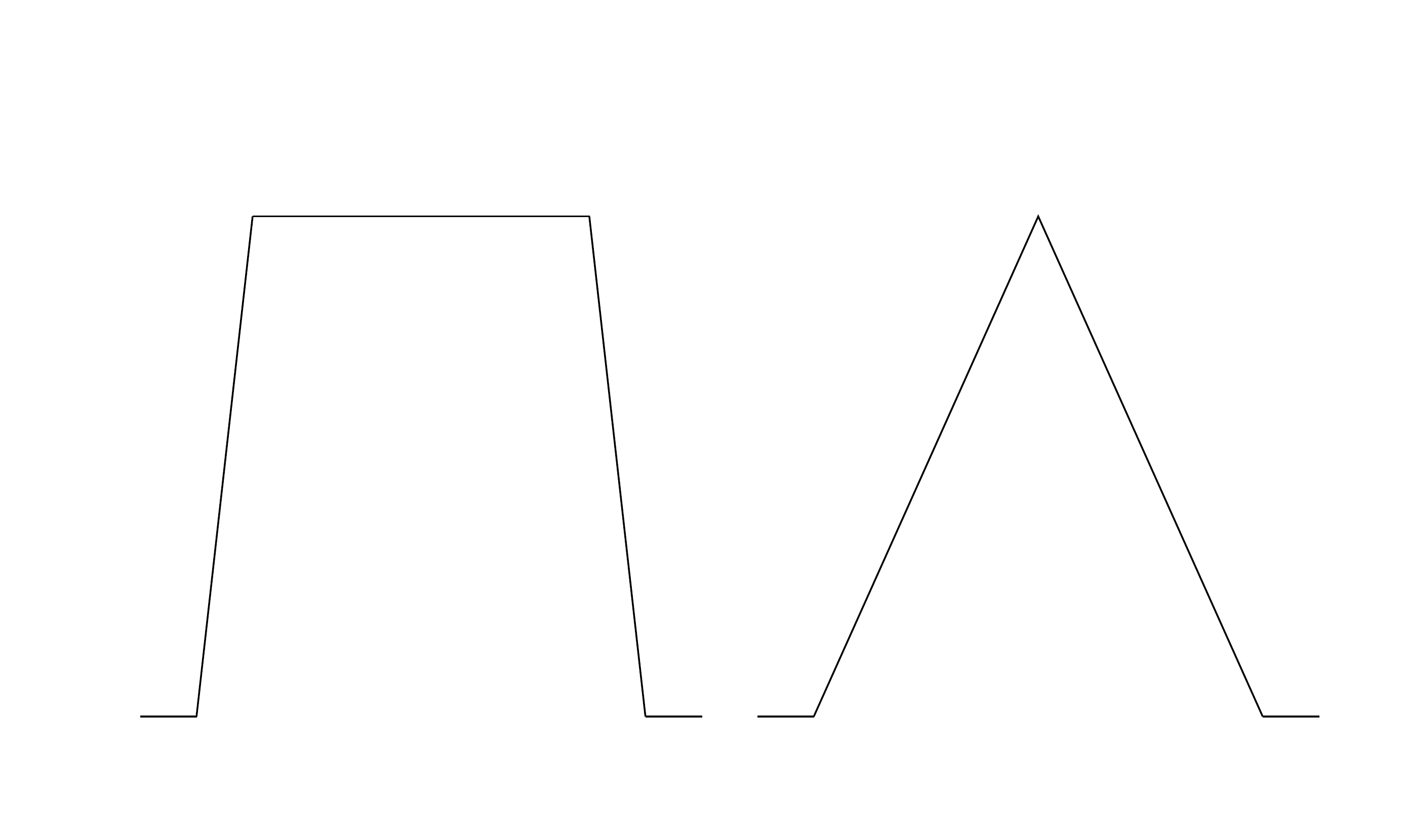}
\caption{The function $\varphi_\delta$ for $p=1$ (left) and $p=\infty$ (right).}
\end{figure}
We take $\{\varrho_\varepsilon\}_{\varepsilon>0}$ a family of standard mollifiers, then for $0<\varepsilon\ll 1$ the function $\varphi_\delta\ast\varrho_\varepsilon$ is admissible. Thus, we get
\[
\pi_1\le \lim_{\varepsilon\to 0^+}\frac{\displaystyle\int_0^1 |(\varphi_\delta\ast \varrho_\varepsilon)'|\,dt}{\displaystyle\int_0^1 |\varphi_\delta\ast \varrho_\varepsilon|\,dt}=\lim_{\varepsilon\to 0^+}\frac{\displaystyle\int_0^1 |(\varphi_\delta')\ast \varrho_\varepsilon|\,dt}{\displaystyle\int_0^1 |\varphi_\delta\ast \varrho_\varepsilon|\,dt}=\frac{\displaystyle\int_0^1 |\varphi_\delta'|\,dt}{\displaystyle\int_0^1 |\varphi_\delta|\,dt}=\frac{2}{1-3\,\delta}.
\]
By taking the limit as $\delta$ goes to $0$, we get the desired conclusion.
\end{proof}

\begin{lemma}
\label{lm:pinfty}
We have
\[
\pi_\infty=\inf_{\varphi\in C^1([0,1])\setminus\{0\}} \left\{\frac{\|\varphi'\|_{L^\infty([0,1])}}{\displaystyle\|\varphi\|_{L^\infty([0,1])}}\, :\, \varphi(0)=\varphi(1)=0\right\}=2.
\]
\end{lemma}
\begin{proof}
We take an admissible test function $\varphi$, then we take $t_0\in(0,1)$ one of the maximum points of $|\varphi|$. We obtain
\[
\|\varphi\|_{L^\infty([0,1])}=|\varphi(t_0)|=|\varphi(t_0)-\varphi(0)|\le \int_0^{t_0} |\varphi'(\tau)|\,d\tau\le t_0\,\|\varphi'\|_{L^\infty([0,1])},
\]
and
\[
\|\varphi\|_{L^\infty([0,1])}=|\varphi(t_0)|=|\varphi(1)-\varphi(t_0)|\le \int_{t_0}^1 |\varphi'(\tau)|\,d\tau\le (1-t_0)\,\|\varphi'\|_{L^\infty([0,1])}.
\]
By taking the product of the last two estimates, we get
\[
\sqrt{\frac{1}{t_0\,(1-t_0)}}\le \frac{\|\varphi'\|_{L^\infty([0,1])}}{\displaystyle\|\varphi\|_{L^\infty([0,1])}}.
\]
By observing that 
\[
\sqrt{\frac{1}{t_0\,(1-t_0)}}\ge 2,\qquad \mbox{ for every } t_0\in(0,1),
\]
we get that $\pi_\infty\ge 2$. 
\par
In order to get the reverse inequality, we fix $0<\delta<1/2$ and take the function
\[
\varphi_\delta=\left\{\begin{array}{rl}
0,& \mbox{ if } 0\le t\le\delta,\\
&\\
\displaystyle1-\frac{\left|t-\dfrac{1}{2}\right|}{\dfrac{1}{2}-\delta},& \mbox{ if } \delta<t<1-\delta,\\
&\\
0,&\mbox{ if } 1-\delta\le t\le 1.
\end{array}
\right.
\] 
By taking as above the convolution with the standard mollifiers $\{\varrho_\varepsilon\}_{\varepsilon>0}$, we get\footnote{In the second inequality, we use that
\[
\int_0^1 \varrho_\varepsilon\,dt=1,\qquad \mbox{ for every }\varepsilon\ll 1.
\]}
\[
\pi_\infty\le \lim_{\varepsilon\to 0^+}\frac{\|\varphi'_\delta\ast \varrho_\varepsilon\|_{L^\infty([0,1])}}{\displaystyle\|\varphi_\delta\ast\varrho_\varepsilon\|_{L^\infty([0,1])}}\le \lim_{\varepsilon\to 0^+} \frac{\|\varphi'_\delta\|_{L^\infty([0,1])}}{\displaystyle\|\varphi_\delta\ast\varrho_\varepsilon\|_{L^\infty([0,1])}}=\frac{\|\varphi'_\delta\|_{L^\infty([0,1])}}{\displaystyle\|\varphi_\delta\|_{L^\infty([0,1])}}=\frac{1}{\dfrac{1}{2}-\delta}.
\]
We can now take the limit as $\delta$ goes to $0$ and obtain that $\pi_\infty\le 2$, as well.
\end{proof}

\end{document}